\documentclass[12pt]{article}
\usepackage{amsthm,amsfonts,amsmath, amscd,dsfont}
\usepackage{mathrsfs}
\usepackage{comment}

\swapnumbers
                              
\theoremstyle{plain}
\newtheorem{thm}{THEOREM}[section]
\newtheorem{lm}[thm]{LEMMA}
\newtheorem{cl}[thm]{COROLLARY}

\theoremstyle{definition}
\newtheorem{defi}[thm]{DEFINITION}
\newtheorem{remark}[thm]{Remark}
\theoremstyle{remark}
\newcommand{\upchi}{\raise1pt\hbox{$\chi$}}
\newcommand{\R}{{\mathord{\mathbb R}}}
\newcommand{\C}{{\mathord{\mathbb C}}}
\newcommand{\Z}{{\mathord{\mathbb Z}}}
\newcommand{\N}{{\mathord{\mathbb N}}}

\newcommand{\tr}{{\rm Tr}}
\renewcommand{\|}{{\Vert}}
\numberwithin{equation}{section}
\newcommand{\un}{{\rm 1\kern -2.5pt l}}
\newcommand{\id}{{\rm Id}}

\newcommand{\one}{{\mathds1}}

\begin{document}
%%%%%%%%DRAFT%%%%%%%
\iffalse
\fi
%%%%%%%%%%%%%%%%%%%%%%
%\markboth{\scriptsize{CJL \version}}{\scriptsize{CJL April 6, 2020}}

\title{{\sc  Monotonicity versions of Epstein's Concavity Theorem and related inequalities}}
\author{
\vspace{5pt}  Eric A. Carlen$^{1}$  and Haonan Zhang$^{2}$ \\
\vspace{5pt}\small{$^{1}$ Department of Mathematics, Hill Center,}\\[-6pt]
\small{Rutgers University,
110 Frelinghuysen Road
Piscataway NJ 08854-8019 USA}\\
\vspace{5pt}\small{$^{2}$ Institute of Science and Technology Austria,
	Am Campus 1, 3400 Klosterneuburg, Austria}\\[-6pt]
 }

%\date{October 27, 2016}
\maketitle

\begin{abstract} Many trace inequalities can be expressed  either as concavity/convexity theorems or as monotonicity theorems. A classic example is the joint convexity of the quantum relative entropy which is equivalent to the Data Processing Inequality.  The latter says that quantum operations can never increase the relative entropy.  The monotonicity versions often have many advantages, and often have direct physical application, as in the example just mentioned. Moreover, the monotonicity results are often valid for a larger class of maps than, say, quantum operations (which are completely positive). In this paper we prove several new monotonicity results, the first of which is a monotonicity theorem that has as a simple corollary  a celebrated concavity theorem of Epstein.  Our starting points are the monotonicity versions of the Lieb Concavity and the Lieb Convexity Theorems. We also give two new proofs of these in their general forms using interpolation. We then prove our new monotonicity theorems by several duality arguments. 

\end{abstract}

\footnotetext [1]{Work partially supported by U.S. National Science Foundation grant DMS 2055282.}

\footnotetext [2]{Work partially supported by the Lise Meitner fellowship, Austrian Science Fund (FWF) M3337.
	 \hfill\break
\copyright\, 2022 by the authors. This paper may be reproduced, in its
entirety, for non-commercial purposes.
}

\medskip

\centerline{Key Words: monotonicity, concavity/convexity, duality, trace inequality, Schwarz maps}

\centerline{ MSC: 39B62, 81P45, 94A17}
%%%%%%%%%%%%%%%%%%%%%%%%%%%%%%%%%%%%%%%%%%%%%%%
%%%%%%%%%%%%%%%%%%%%%%%%%%%%%%%%%%%%%%%%%%%%%%%

\maketitle

\section{Introduction}

Many important concavity and convexity theorems for matrix trace functionals have formulations as monotonicity theorems under some appropriate 
class of positive maps.  In the following, $M_m(\C)$ denotes the $m\times m$ complex matrices,  $M_m^{+}(\C)$, denotes the closed cone in  
$M_m(\C)$ consisting of positive semidefinite matrices, and  $M_m^{++}(\C)$, denotes the open cone in  $M_m(\C)$ consisting of positive 
definite matrices.  A linear  map $\Phi: M_n(\C) \to M_m(\C)$ is positive in case $\Phi(X) \in M_m^{+}(\C)$ whenever $X\in M_n^{+}(\C)$.  
Some more refined notions of positivity will be central here, and these are recalled in the next section. 
To keep the introduction brief, we assume  for now that the notion of complete positivity is familiar to the reader.   For each $n$,  equip $M_n(\C)$ with the Hilbert--Schmidt inner product, making it a Hilbert space. With respect to this Hilbert space structure, every linear $\Phi: M_n(\C) \to M_m(\C)$ has an adjoint, denoted $\Phi^\dagger$, with $\Phi^\dagger:M_m(\C) \to M_n(\C)$. A map  $\Phi: M_n(\C) \to M_m(\C)$ is {\em unital} in case
$\Phi(\one) = \one$, and it is {\em trace preserving} in case $\tr[\Phi(X)] = \tr[X]$ for all $X\in M_n(\C)$. It is easy to see that $\Phi$ is unital if and only if $\Phi^\dagger$ is trace preserving. Completely positive trace preserving maps are known as {\em quantum operations}. Any sort of physically possible manipulation of a quantum state is described by a quantum operation, hence highlighting the importance of this class of positive maps. 

In a fundamental  1973 paper \cite{Lieb73WYD}, Lieb proved a number of  concavity and 
convexity theorems for trace functionals of matrices. The first two of these were

\begin{thm}[Lieb Concavity Theorem]\label{L1}
For $s,t\ge 0$ and $s+t \leq 1$, and any fixed $K\in M_n(\C)$, the function
\begin{equation}\label{lieb1}
(X,Y) \mapsto \tr[K^*Y^s K X^t]
\end{equation}
is jointly concave on $M_n^+(\C)\times M_n^+(\C)$. 
\end{thm}

\begin{thm}[Lieb Convexity Theorem] \label{L2} For all $s,t\geq 0$ and $s+t \leq 1$, 
\begin{equation}\label{lieb2} 
(X,Y,K) \mapsto    \tr[ K^*Y^{-s}  KX^{-t}]
\end{equation} 
is jointly convex on $M_n^{++}(\C)\times M_n^{++}(\C)\times M_n(\C)$.  
\end{thm}

In both cases, by a simple argument originally pointed out by Araki \cite{A75}, the general case $s+t \leq 1$ follows from the special case $s+t =1$; see \cite[Section 2]{Carlen22review} for more information. Lindblad \cite{Lind74} took the further special case in which $K =\one$ and 
$X$ and $Y$ are both non-degenerate  density matrices; i.e., elements of $M_n^{++}(\C)$ with unit trace, to deduce that 
\begin{equation}\label{relentrop}
 D(X||Y)  :=  \tr[X(\log X - \log Y)] = \lim_{t\uparrow 1}\frac{1}{1-t}\left(1 -  \tr[Y^{1-t}X^t]\right)\ 
\end{equation}
is {\em jointly convex} in $X$ and $Y$.   The quantity $ D(X||Y) $ is known as the {\em Umegaki quantum relative entropy} of $X$ with respect to $Y$.  Lindblad  in 1975 \cite{Lind75} then went on to prove the monotonicity version of this result: 
\begin{thm}[Data Processing Inequality]\label{thm:DPI}  For all completely positive trace preserving maps $\Phi^\dagger:M_m(\C) \to M_n(\C)$, and all density matrices $X,Y \in M_m^{++}(\C)$, 
\begin{equation}\label{DPI}
D(\Phi^\dagger(X)||\Phi^\dagger(Y)) \leq D(X ||Y)\ .
\end{equation}
\end{thm}
This inequality, which roughly says that any operation one may perform on two quantum states can only make them harder to distinguish, is one of the cornerstones of quantum information theory.   

It is correct to refer to the Data Processing 
Inequality as a  {\em monotonicity version} of the joint convexity of the quantum relative entropy because it is easy to prove one from the other. 
On the one hand, Lindblad derived \eqref{DPI} from the joint convexity of $D(X||Y)$  using the structure theory for completely positive 
maps provided by the {\em Stinespring Dilation Theorem }\cite{S55}.   On the other hand, one recovers the joint convexity statement by considering 
a very particular quantum operation, namely the partial trace. See \cite{Carlen22review} for more information, though we provide examples of this below.

Already, this example shows the merit in recasting a concavity or convexity theorem as a monotonicity theorem -- {\em often the monotonicity has  
a significant physical interpretation}.  There is however, another advantage of the monotonicity formulations: They often 
hold for a wider class of positive maps extending beyond  the class of completely positive maps, and these more general results 
may no longer be deduced from convexity or concavity statements since the 
 Stinespring Dilation Theorem is then not available.   
 
 The first result of this type is due to Uhlmann \cite[Proposition 17]{Uhlmann77} in 1977 who proved a monotonicity formulation of the Lieb Concavity Theorem. The theorem refers to the class of positive maps  $\Phi:M_n(\C) \to M_m(\C)$  such that for all  $K\in M_n(\C)$, 
 \begin{equation}\label{schwarz1}
 \Phi(K^*K) \geq \Phi(K)^*\Phi(K)\ .
 \end{equation}
 Terminology is not completely standardized, but in this paper we refer to this class of maps as the class of  {\em Schwarz maps}.
 Evidently such a map is positive, and as we recall in the next section, the class of Schwarz maps strictly includes the class of  completely positive maps.
  
 \begin{thm}\label{L1M}
	 For all $0 \leq p \leq 1$, all $m,n\in \N$, all 
	$A,B \in M_m^{+}(\C)$ and all $K\in M_n(\C)$,  and all  Schwarz maps $\Phi: M_n(\C)\to M_m(\C)$,
		\begin{equation}\label{lieb22concave} 
	\tr [\Phi(K)^\ast A^{p} \Phi(K)B^{1-p}]\le \tr [K^\ast \Phi^\dagger(A)^{p} K \Phi^\dagger(B)^{1-p}] \ .
	\end{equation} 
\end{thm}
Assume furthermore that $\Phi$ is unital. Again, taking $K=\one$ and differentiation in $p$ at $p=1$, one obtains \eqref{DPI}, but now under the more general conditions that $\Phi$ is a unital Schwarz map. This generalization of the Data Processing Inequality was made by Uhlmann in 1977 \cite{Uhlmann77}, and for 40 years it remained the state of the art until M\"uller-Hermes and Reeb  \cite{MHR17} proved the ultimate extension, namely that \eqref{DPI} is valid for all unital positive maps $\Phi$. 

Naturally, one may ask if \eqref{lieb22concave} holds for all positive maps $\Phi$. It turns out that the class of  Schwarz maps is optimal, and therefore the inequality \eqref{lieb22concave} characterizes Schwarz maps.

\begin{thm}\label{SMO} If a positive map $\Phi:M_n(\C)\to M_m(\C)$ satisfies \eqref{lieb22concave} for all $0\le p\le 1$, all $A,B \in M_m^{+}(\C)$ and all $K\in M_n(\C)$, then $\Phi$ is a Schwarz map. Thus, satisfaction of the inequality \eqref{lieb22concave} characterizes the class of  Schwarz maps.
\end{thm}

\begin{proof}
Note that at the endpoint case, say $p=0$, \eqref{lieb22concave} becomes 
\begin{equation}\label{endpoint}
\tr [\Phi(K)^\ast  \Phi(K)B]\le \tr [K^\ast K \Phi^\dagger(B)] = \tr [\Phi (K^\ast K) B]\ .
\end{equation}
Since $B$ is any element in $M^+_m(\C)$, this implies the operator inequality \eqref{schwarz1},
showing that the class of Schwarz maps is the broadest class of positive linear maps for which  Theorem~\ref{L1M} is valid. 
\end{proof}

There is also a monotonicity extension of the Lieb Convexity Theorem, Theorem~\ref{L2}:

\begin{thm}\label{L2M}   For all $0 \leq t \leq 1$, all $m,n\in \N$, all 
$X,Y \in M_m^{++}(\C)$, all $K\in M_m(\C)$  and all  Schwarz maps $\Phi: M_n(\C)\to M_m(\C)$  such that $\Phi^\dagger(\one) \in M_n^{++}(\C)$, 
\begin{equation}\label{lieb22} 
 \tr[ \Phi^\dagger(K^*)\Phi^\dagger(Y)^{t-1}  \Phi^\dagger(K)\Phi^\dagger(X)^{-t}]  \leq   \tr[ K^*Y^{t-1}  KX^{-t}] \ .
\end{equation} 
\end{thm}

For unital completely positive maps, this follows from the Lieb Convexity Theorem,  Theorem~\ref{L2},  and the Stinespring Dilation Theorem \cite{S55}, 
as explained in \cite[Section 3] {Carlen22review}. However, the first explicit statement appeared only in 1996, when Petz proved it directly for unital $2$-positive maps, \cite{P96},  a class that includes unital completely positive maps, but is strictly included in the class of unital Schwarz 
maps as 
explained below.  By the very simple Lemma 1 of  \cite{HP12quasi}, Theorem~\ref{L1M} is equivalent to Theorem~\ref{L2M}, and this together with Uhlmann's proof of Theorem~\ref{L1M} proves Theorem~\ref{L2M} with the additional condition that $\Phi$ is unital.    Lieb \cite{Lieb73WYD} proved the equivalence of Theorem~\ref{L1} and  Theorem~\ref{L2}, but the proof of equivalence is simpler in the monotonicity formulation. Theorem~\ref{L2M}  has a significant physical interpretation in terms of {\em monotone metrics}; 
see \cite[Section 8] {Carlen22review}  for more discussion. 

Again, the class of Schwarz maps is optimal,  and therefore the inequality \eqref{lieb22} characterizes Schwarz maps:

\begin{thm}\label{SMO2} If a positive map $\Phi:M_n(\C)\to M_m(\C)$ satisfies \eqref{lieb22} for all $0\le t\le 1$, all $X,Y \in M_m^{+}(\C)$ and all $K\in M_n(\C)$, then $\Phi$ is a Schwarz map.   Thus, satisfaction of the inequality \eqref{lieb22} characterizes the class of  Schwarz maps.
\end{thm}

\begin{proof}In the endpoint $t=0$ case, \eqref{lieb22} becomes
\begin{equation}\label{tracialA}
\tr\left[\Phi^\dagger(K^\ast)\Phi^\dagger(Y)^{-1}\Phi^\dagger(K)\right]\le \tr[ K^*Y^{-1} K] \ .
\end{equation}
By a recent result \cite[Theorem 4]{CMH22}, the inequality \eqref{tracialA}  characterizes Schwarz maps. 
\end{proof}

\begin{remark}
	That $\Phi$ is Schwarz is equivalent to \eqref{tracialA} was proved recently in \cite{CMH22}. In a preliminary version of this paper, we also obtained this characterization, in the class of unital maps,  using results in an earlier version of \cite{CMH22} and the work of Hiai and Petz \cite{HP12quasi}. Although the argument made use of an additional hypothesis, namely unitality, the argument may still be of interest, and so we sketch it here.
	
	The method of Hiai and Petz in \cite{HP12quasi} uses the integral representation for operator monotone increasing functions $f$: 
	$$f(x) :=  \beta + \gamma x + \int_{(0,\infty)}   \frac{x}{\lambda + x}  (1 + \lambda)   {\rm d}\mu(\lambda)\ ,$$
	where $\mu$ is a finite Borel measure on $(0,\infty)$. The key observation underlying our characterization of unital Schwarz maps is that in the integral representation of 
	$$
	x^r =  \int_{(0,\infty)}   \frac{x}{\lambda + x}  (1 + \lambda)   {\rm d}\mu(\lambda)\   \quad{\rm where}\quad {\rm d}\mu(\lambda) = \frac{\sin(\pi r)}{\pi} \frac{\lambda^{r-1}}{1+\lambda}{\rm d}\lambda\ ,
	$$
	with $0<r<1$, one does not have the affine term, i.e. $\beta = \gamma = 0$.  In the argument of Hiai and Petz, these terms were treated using the Schwarz inequality.
	The remaining  integral term only required a {\em tracial inequality} equivalent to \eqref{tracialA} when $\Phi$ is unital. This tracial inequality was shown in an earlier version \cite{CMH22} to be valid whenever $\Phi$ is a unital Schwarz map. On the other hand, we observed that when the tracial inequality holds, the method of Hiai and Petz yields \eqref{lieb22concave} without assuming that $\Phi$ is a Schwarz map since the affine terms are not present. But then by Theorem~\ref{SMO}, $\Phi$ is a Schwarz map. 
\end{remark}

A final advantage of the monotonicity formulations is that they are often easier to directly prove than the corresponding convexity or 
concavity theorems, which one then readily recovers with a simple partial trace argument as we demonstrate below. Indeed, 
a very simple proof of Theorems~\ref{L1M} and~\ref{L2M} was recently given in \cite{CMH22}, building on previous work of Hiai and Petz \cite{HP12quasi}, 
extending their argument to the wider generality of Schwarz maps. Later in the paper, we give a new proof of 
Theorem~\ref{L1M} that is modeled on Lieb's original proof of Theorem~\ref{L1} by interpolation.  
The interpolation proof of the monotonicity result, which is different from Uhlmann's interpolation argument \cite{Uhlmann77}, seems especially simple to us. We also use interpolation to give another proof of Theorem~\ref{L2M}.

The main theme of this paper is to obtain new monotonicity theorems from Theorems~\ref{L1M} and \ref{L2M}, among others, by a duality technique 
involving Legendre transforms or variations upon them. This is a technique that both authors have successfully applied to prove new convexity and concavity results. One of us recently used this approach  to prove a monotonicity version of Lieb's theorem \cite[Theorem 6]{Lieb73WYD} stating that  for all self-adjoint $H\in M_n(\C)$, 
\begin{equation}\label{Lcon2}
X \mapsto \tr[ \exp(H + \log X)]
\end{equation}
is concave on $M_n^{++}(\C)$; see \cite{Carlen22lieb}.

Our first new result is a monotonicity version of a theorem of Epstein \cite{Epstein73}:

\begin{thm}\label{Ep1}  Let $0 < p \leq 1$, and let $\Phi: M_m(\C) \to M_n(\C)$ be a unital Schwarz map.  Then for any 
$A\in M_{n}^{++}(\C)$ and any $B\in M_m(\C)$,
\begin{equation}\label{EpMon}
 \tr[ (\Phi(B)^*A^p \Phi(B))^{1/p}]  \leq   \tr[ (B^*\Phi^\dagger(A)^p B)^{1/p}]  \ .
\end{equation}
\end{thm}

Below we give two proofs of this theorem, each of which lends itself to different generalizations. First consider the special case in which $n=2m$ and 
\begin{equation}\label{embed}\Phi(X) = \left[\begin{array}{cc} X & 0 \\ 0 & X\end{array}\right] \quad{\rm so\ that}\quad \Phi^\dagger\left(\left[\begin{array}{cc} A & B\\C & D\end{array}\right]   \right)  = A+D\ .
\end{equation}
Clearly $\Phi$ is unital completely positive and hence Schwarz, as explained below. 
Now consider $A_1,A_2\in M_{m}^{+}(\C)$, $B\in M_m(\C)$ and define $A := \left[\begin{array}{cc} A_1 & 0 \\ 0 & A_2\end{array}\right]$. 
Then 
$$
 \tr[ (\Phi(B)^*A^p \Phi(B))^{1/p}]  = \tr\left[ \begin{array}{cc} (B^*A_1^pB)^{1/p} & 0 \\ 0 & (B^*A_2^p B)^{1/p} \end{array}\right] = \tr[ (B^*A_1^pB)^{1/p}] + 
 \tr[ (B^*A_2^pB)^{1/p}] \ ,
$$
and 
$$
\tr[ (B^*\Phi^\dagger(A)^p B)^{1/p}]   = \tr[(B^*(A_1+A_2)^pB)^{1/p}]\ .
$$
Thus, by homogeneity,
$$
\frac12 \tr[ (B^*A_1^pB)^{1/p}] + \frac12  \tr[ (B^*A_2^pB)^{1/p}]  \leq  \tr\left[\left(B^*\left(\frac{A_1+A_2}{2}\right)^pB\right)^{1/p}\right]\ .
$$
Together with the obvious continuity, this proves:

\begin{cl}[Epstein's Theorem \cite{Epstein73}] For any $B\in M_n(\C)$ and all $0 \leq  p \leq  1$, the map
\begin{equation}\label{epcon}
A \mapsto \tr[ (B^* A^p B)^{1/p}]
\end{equation}
is concave on $M_n^{+}(\C)$. 
\end{cl} 

\begin{remark}  In the endpoint case $p=1$, \eqref{EpMon} reduces to $\tr[ \Phi(B)\Phi(B)^*A ]  \leq   \tr[ BB^*\Phi^\dagger(A) ]$, and again since $A\in M_n^{+}$ is arbitrary, $\Phi$ satisfies \eqref{schwarz1}. Hence the validity of \eqref{EpMon} for $0 <  p \leq 1$ is characteristic of the class of Schwarz maps, which is the broadest class of unital positive linear maps for which Theorem~\ref{Ep1} is valid. 
\end{remark}

\subsection*{Acknowledgement} We are grateful to the anonymous referee for pointing out a number of typos and helpful comments.

\section{Positivity and duality}  

In this section, we recall some different notions of positivity, and some duality lemmas we shall use in later sections. 

Consider a linear map $\Phi: M_n(\C) \to M_m(\C)$. We identify $M_n(\C)\otimes M_2(\C)$ with the $2\times 2$ block matrices
$\left[\begin{array}{cc} A & B\\C & D\end{array}\right]$ and then 
\begin{equation}\label{2pos}
\Phi\otimes \id\left(   \left[\begin{array}{cc} A & B\\C & D\end{array}\right]   \right) = \left[\begin{array}{cc} \Phi(A) & \Phi(B)\\\Phi(C) & \Phi(D)\end{array}\right]\ .
\end{equation}
 The map $\Phi$ is defined to be $2$-{\em positive} in case the matrix on the right in \eqref{2pos} belongs to $M_{2m}^{+}(\C)$ whenever
$\left[\begin{array}{cc} A & B\\C & D\end{array}\right] \in M_{2n}^{+}(\C)$. The notion of {\em $k$-positivity} is defined in 
the analogous way for all positive integers $k\geq2$, and the map $\Phi$ is {\em completely positive}  in case it is  $k$-positive for all $k$. 
By a theorem of Choi \cite{Choi74}, 
which generalizes an earlier result of Kadison \cite{K52},  whenever $\Phi$ is a unital $2$-positive map, for all $X\in M_n(\C)$,
\begin{equation}\label{choiS}
\Phi(X^*X) \geq \Phi(X)^*\Phi(X)\ .
\end{equation}
A {\em Schwarz map} is a linear map satisfying \eqref{choiS}. Evidently, every Schwarz map is positive. 
Choi also showed \cite{Choi80} that while every $2$-positive map is a Schwarz map, the converse is not true:  There are Schwarz maps that are not $2$-positive. 
The notion of a {\em generalized Schwarz map} may be found in \cite{CMH22}: The map $\Phi: M_n(\C) \to M_m(\C)$, not necessarily unital,  is a {\em generalized Schwarz map} in case
\begin{equation}\label{genschw}
\left[ \begin{array}{cc} \Phi(\one) & \Phi(X) \\ \Phi(X)^* &\Phi(X^*X)\end{array}\right] \geq 0\ .
\end{equation}
Using the well-known fact that  for $A \in M_n^{+}(\C)$, with $A^+$ denoting its Moore--Penrose generalized inverse,
\begin{equation}\label{pchar}
\left[\begin{array}{cc} A & Z\\Z^\ast & B\end{array}\right]  \geq 0  \iff B \geq Z^*A^+Z\ ,
\end{equation}
one sees, as noted in  \cite{CMH22} that every $2$-positive map is a generalized Schwarz map, and that a unital map is Schwarz if and only if it is generalized Schwarz. Thus we have the following hierarchy of classes of positive maps:
\begin{equation}\label{positmaps}
{\rm Completely\ Positive} \subset {\rm 2-Positive} \subset {\rm Generalized\ Schwarz} \subset {\rm Positive}\ .
\end{equation}

We shall also use some  duality lemmas.  The first of these depends on the reverse H\"older inequality for traces that we briefly recall for completeness; see \cite{Ch21} and \cite[Appendix]{HMPB11rmp} for a more general result without the positivity condition that meets our needs here. 
Let $X,Y\in M_n^{++}(\C)$.
We take the trace of $XY$ using a basis in which $X_{i,j} = \delta_{i,j}\lambda_j$ and $Y_{j,i} =\sum_{k} U_{j,k}\mu_k \overline{U_{i,k}}$.  Then
$$
\tr[XY] = \sum_{j,k} \lambda_j|U_{j,k}|^2 \mu_k\ .
$$
 The matrix $S$ with $S_{j,k}:= |U_{j,k}|^2$ is doubly stochastic, and hence by Birkhoff's Theorem \cite{B46}, $S$ is a convex combination of permutation matrices
$P^{(\pi)}$, $\pi\in {\mathcal S}_n$ with ${\mathcal S}_n$ the symmetric group on $n$ letters. 
Hence $\tr[XY]$ is a convex combination of sums of the form
$\sum_{j=1}^n \lambda_j \mu_{\pi(j)}$, $\pi \in {\mathcal S}_n$.
To each of these we may apply the classical reverse H\"older inequality to conclude that for $0< r < 1$ or $-\infty < r < 0$,
$$
\sum_{j=1}^n \lambda_j \mu_{\pi(j)}  \geq \left( \sum_{j=1}^n  \lambda_j^r  \right)^{1/r} \left(  \sum_{j=1}^n \mu_{\pi(j)}^{r/(r-1)} \right)^{(r-1)/r} = 
\left(\tr[X^r]\right)^{1/r}\left(\tr[Y^{r/(r-1)}]\right)^{(r-1)/r}\ .
$$
We conclude that  for $X,Y\in M_n^{++}(\C)$:  If $0< r < 1$ or $-\infty < r < 0$, then
\begin{equation}\label{reverseHold}
\tr[XY] \geq  \left(\tr [X^r]\right)^{1/r}  \left(\tr [Y^{r/(r-1)}]\right)^{(r-1)/r}\ .
\end{equation}

\begin{lm}Let $X\in M_n^{+}(\C)$. Then for $1 <  r  < \infty$, 
\begin{equation}\label{maxlem} \left(\tr [X^r]\right)^{1/r} = \max\left\{ \tr[XY]\ :\ Y\in M_n^{+}(\C)\ ,\quad \tr[Y^{r/(r-1)}] =1\ \right\} \ ,
\end{equation}
while for $0 < r < 1 $ or $-\infty < r < 0$,  and $X\in M_n^{++}(\C)$,
\begin{equation}\label{minlem1} \left(\tr [X^r]\right)^{1/r} = \min\left\{ \tr[XY]\ :\ Y\in M_n^{++}(\C)\ ,\quad \tr[Y^{r/(r-1)}] =1\ \right\} \ .
\end{equation}
\end{lm}

\begin{proof} 
First, consider $1 < r < \infty $.     For $Y := \left(\tr [X^r]\right)^{(1-r)/r} X^{r-1}$, we have 
\begin{equation*}
\tr[XY] =   \left(\tr [X^r]\right)^{1/r}\qquad{\rm and}\qquad  \quad \tr[Y^{r/(r-1)}] =1 \ .
\end{equation*}
Then by H\"older's inequality for traces \cite[Section 7.1]{Carlen10notes}, 
$$\tr[XY] \leq  \left(\tr [X^r]\right)^{1/r}  \left(\tr [Y^{r/(r-1)}]\right)^{(r-1)/r},$$
and this proves \eqref{maxlem}.

Next, suppose $X\in M_n^{++}(\C)$, and  $0 < r < 1$  or $-\infty < r < 0$ .  Again with  $Y := \left(\tr [X^r]\right)^{(1-r)/r} X^{r-1}$, we have
\begin{equation}\label{oneside}
\tr[XY] =   \left(\tr [X^r]\right)^{1/r} \qquad{\rm and}\qquad  \quad \tr[Y^{r/(r-1)}] =1\ .
\end{equation}
Now the reverse H\"older inequality \eqref{reverseHold}
together with the example in \eqref{oneside} proves \eqref{minlem1}.
\end{proof}

\section{Proofs of  Theorem~\ref{Ep1} and related inequalities}

\subsection{Proofs of  Theorem~\ref{Ep1}}
\begin{proof}[First proof of Theorem~\ref{Ep1}] When $p=1$, the inequality follows directly from the assumption that $\Phi$ is Schwarz. Now let $0<p<1$ and put $r:= 1/p > 1$, so that
$$
( \tr[ (\Phi(B)^*A^p \Phi(B))^{1/p}])^p =  \left(\tr[(\Phi(B)^*A^p \Phi(B))^r]\right)^{1/r}\ .
$$ 
  Then by \eqref{maxlem}
\begin{eqnarray}\label{dv1}
  \left(\tr[(\Phi(B)^*A^p \Phi(B))^r]\right)^{1/r}
   &=& \max\left\{ \tr[\Phi(B)^*A^p \Phi(B)Y]\ :\ Y\in M_n^{+}(\C)\ ,\quad \tr[Y^{r/(r-1)}] =1\ \right\} \nonumber\\
 &=& \max\left\{ \tr[\Phi(B)^*A^p \Phi(B)Y]\ :\ Y\in M_n^{+}(\C)\ ,\quad \tr[Y^{1/(1-p)}] =1\ \right\} \nonumber\\
  &=& \max\left\{ \tr[\Phi(B)^*A^p \Phi(B)Y^{1-p}]\ :\ Y\in M_n^{+}(\C)\ ,\quad \tr[Y] =1\ \right\}
\end{eqnarray}
By  Theorem~\ref{L1M}, 
$ \tr[\Phi(B)^*A^p\Phi(B)Y^{1-p}]  \leq  \tr[(B^*\Phi^\dagger(A)^pB\Phi^\dagger(Y)^{1-p}]$.
Therefore,
\begin{eqnarray}\label{dv2}
(\tr[ (\Phi(B)^*A^p \Phi(B))^{1/p}])^p &\leq&  \max\left\{ \tr\left[B^*\Phi^\dagger(A)^pB\Phi^\dagger(Y)^{1-p}\right]\ :\ Y\in M_n^{+}(\C)\ ,\quad \tr[Y] =1\ \right\}\nonumber\\
&\leq&  \max\left\{ \tr\left[B^*\Phi^\dagger(A)^p BZ^{1-p}\right]\ :\ Z\in M_m^{+}(\C)\ ,\quad \tr[Z] =1\ \right\}\nonumber\\
&=&  \max\left\{ \tr\left[B^*\Phi^\dagger(A)^p BZ\right]\ :\ Z\in M_m^{+}(\C)\ ,\quad \tr[Z^{1/(1-p)}] =1\ \right\}\nonumber\\
&=& (\tr[ (B^*\Phi^\dagger(A)^p B)^{1/p}])^p  \ ,
\end{eqnarray}
where in the second inequality we have used the fact that $\Phi^\dagger$ is positive trace preserving and the definition of $\max$, and in the last line we have used \eqref{maxlem}  once again. Now take the $p$-th root of both sides. 
\end{proof}

We now give a second proof of Theorem~\ref{Ep1} that uses a different but related variational formula. The starting point is the generalized H\"older's inequality for traces. For $0 < p < \infty$, and $X\in M_n(\C)$, define $\|X\|_p := (\tr[(X^*X)^{p/2}])^{1/p}$. For $p \geq 1$, this is the Schatten $p$-norm, and otherwise it is a quasi-norm. Then one may prove using, for example, the theory of majorization, that for $X,Y\in M_n(\C)$ and
for $p_i>0,i=0,1,2$ such that $1/p_0=1/p_1 +1/p_2$  \cite[Exercise IV.2.7]{Bhatia97book}
\begin{equation}\label{generalized Holder}
\|XY\|_{p_0}\le\|X\|_{p_1}\|Y\|_{p_2}\ .
\end{equation}
As proved in \cite{Zhang20}, this gives us more variational formulae that are very useful in proving the concavity/convexity theorems. In fact, for any $X,Y,Z\in M_n(\C)$ with $Y$ invertible, \eqref{generalized Holder} yields (here we use the notation $|A|:=(A^\ast A)^{1/2}$)
\begin{equation*}
\tr[|XZ|^{p_0}]\le \|XY\|_{p_1}^{p_0}\|Y^{-1}Z\|^{p_0}_{p_2}\le \frac{p_0}{p_1}\tr\left[|XY|^{p_1}\right]+\frac{p_0}{p_2}\tr\left[|Y^{-1}Z|^{p_2}\right]\ ,
\end{equation*}
or equivalently
\begin{equation*}
\tr\left[|XY|^{p_1}\right]\ge \frac{p_1}{p_0}\tr\left[|XZ|^{p_0}\right]-\frac{p_1}{p_2}\tr\left[|Y^{-1}Z|^{p_2}\right] \ .
\end{equation*}
One can find $Z\in M_n(\C)$ to saturate the inequality so that \cite[Theorem 3.3]{Zhang20}
\begin{align*}
\tr\left[|XY|^{p_1}\right]
&=\max_{Z\in M_n(\C)}\left\{\frac{p_1}{p_0}\tr\left[|XZ|^{p_0}\right]-\frac{p_1}{p_2}\tr\left[|Y^{-1}Z|^{p_2}\right]\right\}\\
&=\max_{H\in M_n^{+}(\C)}\left\{\frac{p_1}{p_0}\tr\left[(XHX^\ast)^{p_0/2}\right]-\frac{p_1}{p_2}\tr\left[(|Y^{-1}HY^{\ast-1})^{p_2/2}\right]\right\} \ .
\end{align*}

Now for any $0<p<1$, take $(X,Y)=(B^\ast, A^{p/2})$ with $A\in M_n^{++}(\C)$ and $(p_0,p_1,p_2)=(2,2/p,2/(1-p))$:
\begin{equation}\label{dual2}
\tr\left[\left(B^{\ast}A^{p}B\right)^{1/p}\right]
=\max_{H\in M_n^{+}(\C)}\left\{\frac{1}{p}\tr [B^\ast H B]-\frac{1-p}{p}\tr\left[\left(A^{-p/2}H A^{-p/2}\right)^{1/(1-p)}\right]\right\} \ .
\end{equation}

\begin{proof}[Second proof of Theorem~\ref{Ep1}] We   apply \eqref{dual2} in conjunction with the monotonicity of {\em sandwiched $\alpha$-R\'enyi relative entropy} for $\alpha>1$; see M\"uller-Hermes and Reeb \cite[Theorem 2]{MHR17} and Beigi \cite{Beigi13}. Recall that the monotonicity result of M\"uller-Hermes and Reeb \cite[Theorem 2]{MHR17} states: For $\alpha>1$,
	\begin{equation}\label{sandwiched}
	\tr\left[ \left(\Lambda(\sigma)^{(1-\alpha)/(2\alpha)}\Lambda(\rho)\Lambda(\sigma)^{(1-\alpha)/(2\alpha)}\right)^{\alpha}\right]
	\le \tr\left[ \left(\sigma^{(1-\alpha)/(2\alpha)}\rho\sigma^{(1-\alpha)/(2\alpha)}\right)^{\alpha}\right]\ ,
	\end{equation}
	for any positive trace preserving map $\Lambda:M_n(\C)\to M_m(\C)$ and any $\rho,\sigma\in M_n^+(\C)$.
	
	If $\Phi$ is unital Schwarz, then
\begin{align*}
&\tr\left[\left(\Phi(B)^{\ast}A^{p}\Phi(B)\right)^{1/p}\right]\\
=&\max_{H\in M_n^{+}(\C)}\left\{\frac{1}{p}\tr \left[\Phi(B)^\ast H \Phi(B)\right]-\frac{1-p}{p}\tr\left[\left(A^{-p/2}H A^{-p/2}\right)^{1/(1-p)}\right]\right\}\\
\le &\max_{H\in M_n^{+}(\C)}\left\{\frac{1}{p}\tr \left[\Phi(B)^\ast H \Phi(B)\right]-\frac{1-p}{p}\tr\left[\left(\Phi^\dagger(A)^{-p/2}\Phi^\dagger(H)\Phi^\dagger(A)^{-p/2}\right)^{1/(1-p)}\right]\right\}\\
\le &\max_{H\in M_n^{+}(\C)}\left\{\frac{1}{p}\tr \left[B^\ast \Phi^\dagger(H) B\right]-\frac{1-p}{p}\tr\left[\left(\Phi^\dagger(A)^{-p/2}\Phi^\dagger(H)\Phi^\dagger(A)^{-p/2}\right)^{1/(1-p)}\right]\right\}\\
\le &\max_{H\in M_m^{+}(\C)}\left\{\frac{1}{p}\tr \left[B^\ast H B\right]-\frac{1-p}{p}\tr\left[\left(\Phi^\dagger(A)^{-p/2}H\Phi^\dagger(A)^{-p/2}\right)^{1/(1-p)}\right]\right\}\\
=&\tr\left[\left(B^{\ast}\Phi^\dagger(A)^{p}B\right)^{1/p}\right]\ ,
\end{align*}
where we have used \eqref{dual2} in the two equalities, the monotonicity of sandwiched $\alpha$-R\'enyi relative entropy \eqref{sandwiched} with $\alpha=1/(1-p)>1$ in the first inequality, the fact that $\Phi$ is Schwarz in the second inequality, and the definition of $\max$ in the third inequality. 
\end{proof}

\subsection{Generalizations of Theorem~\ref{Ep1}}

In Lieb Concavity Theorem, one does not require the exponents to be $t$ and $1-t$. Theorem~\ref{L1}  holds, 
as stated, when $0 < s,t$ and $s+t \leq 1$. There is a corresponding version of the monotonicity version, but for $s+t < 1$, 
we require $\Phi$ to be {\em semiunital}.  

\begin{defi}\label{sesquidef} A linear map $\Phi: M_m(\C) \to M_n(\C)$ is {\em semiunital} in case
\begin{equation}\label{sesqueq}
\Phi^\dagger(\one)  = \frac{n}{m}\one\ .
\end{equation}
It is {\em sesquiunital} in case it is unital and semiunital. 
\end{defi}

For any unital map $\Phi: M_m(\C) \to M_n(\C)$, $\Phi^\dagger$ is trace preserving, so that $n = \tr[\Phi^\dagger(\one)]$ and hence  if 
$\Phi^\dagger(\one) = c \one$ for some $c$, then necessarily $c= n/m$.   As a fundamental example, consider the unital map $\Phi$ defined in \eqref{embed}. 
In this case $n=2m$ so that $c := n/m =2$ no matter how large $m$ may be. 
Note that \eqref{embed} is precisely the map that was used to recover Epstein's Concavity Theorem from Theorem~\ref{Ep1}.
In the same way, we shall be able to recover the full statement of Lieb Concavity Theorem from the extension of Theorem~\ref{L1M} that we next state and prove:

\begin{thm}\label{L1MB}  Let $0 < s,t$, $s+t \leq 1$, and let $\Phi: M_m(\C) \to M_n(\C)$ be  a semiunital Schwarz map. Then for all 
	$X,Y \in M_n^{+}(\C)$ and all $K\in M_m(\C)$,  
		\begin{equation}\label{liebmon2} 
	\tr [\Phi(K)^\ast X^{s} \Phi(K)Y^{t}]\leq \left(\frac{n}{m}\right)^{1-(s+t)} \tr [K^\ast \Phi^\dagger(X)^{s} K \Phi^\dagger(Y)^{t}] \ .
	\end{equation} 
\end{thm}

\begin{proof}  In case $s+t =1$, the theorem is already proved under the weaker assumption that $\Phi$ is a Schwarz map. Therefore suppose that $t < 1-s$ and define $0 < r < 1$ so that $t = r(1-s)$.   Define $Z = Y^{r}$. Then  by the definitions and Theorem~\ref{L1M},
$$
\tr [\Phi(K)^\ast X^{s} \Phi(K)Y^{t}] = \tr [\Phi(K)^\ast X^{s} \Phi(K)Z^{1-s}] \leq \tr[ K^* \Phi^\dagger(X)^{s} K \Phi^\dagger(Z)^{1-s}]\ .
$$
Therefore,
$$
\frac{m}{n} \tr [\Phi(K)^\ast X^{s} \Phi(K)Y^{t}]  \leq \tr\left[ K^* \left(\frac{m}{n} \Phi^\dagger(X)\right)^{s} K \left(\frac{m}{n} \Phi^\dagger(Y^r)\right)^{1-s}\right].
$$
Since $\frac{m}{n}\Phi^\dagger$ is unital, an operator Jensen type inequality for unital positive maps \cite[Theorem 2.1]{Choi74} (see also  \cite[Proposition 2.7.1]{Bhatia09pdm}) yields
${\displaystyle \frac{m}{n}  \Phi^\dagger(Y^r)\leq \left( \frac{m}{n}  \Phi^\dagger(Y)\right)^r}$. Then by the operator monotonicity of $A\mapsto A^{1-s},0<s<1$:
$$
\left(\frac{m}{n}  \Phi^\dagger(Y^r)\right)^{1-s}\leq \left( \frac{m}{n}  \Phi^\dagger(Y)\right)^{t} \ .
$$
 This finishes the proof. 
\end{proof}

Now applying the method used in our first proof of Theorem~\ref{Ep1}, we prove:

\begin{thm}\label{Ep2}  Let $0 < s \leq 1$ and $s\leq r\le 1$. Let $\Phi$ be any sesquiunital Schwarz map from $M_m(\C)$ to $M_n(\C)$ for  $m,n\in \N$. 
Then for all $A \in M_n^+(\C)$ and all $B\in M_m(\C)$, 
		\begin{equation}\label{ext2} 
	\tr[(\Phi(B)^* A^s\Phi(B))^{r/s}]  \leq \left(\frac{n}{m}\right)^{1-r} \tr[(B^*\Phi^\dagger(A)^s B)^{r/s}]\ .
	\end{equation} 
\end{thm}

Before proving the theorem, we first extract the concavity theorem that is its corollary. This concavity result was proved in \cite[Theorem 1.1]{CL08} and \cite[Theorem 4.1]{Hiai13}.

\begin{cl}\label{cor3.4}  For all $B\in M_m(\C)$, $0 < s < 1$ and $0 < r\le 1$, the map $A \mapsto \tr[(B^* A^s B)^{r/s}]$ is concave on $M_m^{+}(\C)$. 
\end{cl}

\begin{proof} First suppose that $s\leq r\le 1$ so that we may apply Theorem~\ref{Ep2}. As in the discussion below the statement of Theorem~\ref{Ep1}, consider $A_1,A_2\in M_{m}^{+}(\C)$, $B\in M_m(\C)$ and define $A := \left[\begin{array}{cc} A_1 & 0 \\ 0 & A_2\end{array}\right]$. Define $\Phi$ as in \eqref{embed}, and note that $\Phi$ is Schwarz and sesquiunital. 
One computes
$$
\tr[(\Phi(B)^* A^s\Phi(B))^{r/s}]  = \tr[(B^*A_1^sB)^{r/s}] + \tr[(B^*A_2^sB)^{r/s}] \ ,
$$
and 
$$
  \tr[(B^*\Phi^\dagger(A)^s B)^{r/s}] = \tr[ (B^*(A_1+A_2)^sB)^{r/s}]\ .
$$
In this case, $n= 2m$ so that $(n/m)^r = 2^r$. Therefore, by Theorem~\ref{Ep2}
\begin{eqnarray*}
\frac12\left(\tr[(B^*A_1^sB)^{r/s}] + \tr[(B^*A_2^sB)^{r/s}] \right) &=&  \frac12 \tr[(\Phi(B)^* A^s\Phi(B))^{r/s}] \\
&\leq& \left(\frac12\right)^{r} \tr[(B^*\Phi^\dagger(A)^s B)^{r/s}]\\
&=& \tr\left [ \left(B^*\left(\frac{A_1+A_2}{2}\right)^sB\right)^{r/s}\right] \ .
\end{eqnarray*}
As before, with the obvious continuity, midpoint concavity implies concavity.

Next, suppose that $0 < r < s$.  (This part is elementary, and does not depend on Theorem~\ref{L1M}.) Define $p := r/s$, and note that $0 < p < 1$.  Then $(\tr[(B^* A^s B)^{r/s}])^{s/r} = (\tr[(B^* A^s B)^{p}])^{1/p} = \|B^* A^s B\|_p$.	By \eqref{minlem1},
$$\|B^* A^s B\|_p = \min\{ \tr[ B^* A^s B Y] \ : Y\in M_n^{++}(\C)\ ,\ \tr[Y^{p/(p-1)}] = 1  \}\ ,$$
and then since $A \mapsto A^s$ is concave, so is $A \mapsto ( \tr[(B^* A^s B)^{r/s}])^{s/r}$.  Define  $f(A) := (\tr[(B^* A^s B)^{r/s}])^{s/r}$ and 
$g(A) := \tr[(B^* A^s B)^{r/s}]$ so that $g(A) = (f(A))^p$.  Since $f$ is concave, and $h(x):= x^p$ is concave and monotone increasing on $[0,\infty)$, $g$ is also concave.
\end{proof} 

\begin{remark} Note that we have recovered a dimension independent concavity theorem from a monotonicity theorem that may at first appear to be dimension dependent.   But of course, the dimensions $m$ and $n$ enter the inequality provided by the monotonicity theorem only through the ratio $n/m$, and this can be finite or constant in meaningful examples even as the dimensions tend to infinity. 
\end{remark}

The same argument also yields the full Lieb Concavity Theorem from its monotonicity version Theorem \ref{L1MB}.

\begin{proof}[Proof of Theorem~\ref{Ep2}]
We proceed exactly as in the first proof of Theorem~\ref{Ep1}. By continuity we may assume that $0 < s < 1$ and $s <r< 1$, so $0<t < 1$ and $s+t < 1$ for $t:=(r-s)/r$. By \eqref{maxlem}, we have 
$$
(\tr[(\Phi(B)^* A^s\Phi(B))^{r/s}])^{s/r} = \max\{ \tr[\Phi(B)^*A^s\Phi(B)Y^{(r-s)/r}] \ :\ \tr[Y] =1\ \}\ .
$$
By Theorem~\ref{L1MB},  for all $Y\in M_n^+(\C)$ with $\tr[Y] =1$, 
$$
\tr[\Phi(B)^*A^s\Phi(B)Y^{(r-s)/r}] \leq  \left(\frac{n}{m}\right)^{s(1-r)/r}\tr[B^*\Phi^\dagger(A)^s B \Phi^\dagger(Y)^{(r-s)/r}]\ .
$$
Therefore, since $\Phi^\dagger$ is trace preserving and using \eqref{maxlem}, we have
\begin{eqnarray*}
(\tr[(\Phi(B)^* A^s\Phi(B))^{r/s}])^{s/r}  &\leq &  \left(\frac{n}{m}\right)^{s(1-r)/r}\max\{  \tr[B^*\Phi^\dagger(A)^s B Z^{(r-s)/r}]  \ :\ \tr[Z] =1 \}\\
&=&  \left(\frac{n}{m}\right)^{s(1-r)/r}\left( \tr[(B^*\Phi^\dagger(A)^s B)^{r/s}] \right)^{s/r}\ .
\end{eqnarray*}
Hence
$$
(\tr[(\Phi(B)^* A^s\Phi(B))^{r/s}])^{s/r}   \leq  \left(\frac{n}{m}\right)^{s(1-r)/r}\left( \tr[(B^*\Phi^\dagger(A)^s B)^{r/s}] \right)^{s/r}\ ,
$$
which is the same as \eqref{ext2}. 
\end{proof}

%
%Because of homogeneity, the trace-preserving assumption of $\Phi$ is necessary when $p+q<1$. To see this, take the Schwarz map $\Phi:M_n(\C)\mapsto M_{2n}(\C)$ given by 
%$$
%\Phi(X):=
%\begin{bmatrix}
%X &0\\
%0 &X
%\end{bmatrix}
% \qquad{\rm and \quad thus}\qquad  \quad
% \Phi^\dagger\left(
% \begin{bmatrix}
% X&0\\
% 0&Y
% \end{bmatrix}
%  \right)
% =X+Y.
%$$
%Take any $B\in M_n^{++}(\C)$ and choose $A=Y=\Phi(B)$. Then 
%$$
%2\tr [B^{p+q+2}]=\tr[\Phi(B)^*A^p\Phi(B)Y^{q}]  \leq  \tr[B^*\Phi^\dagger(A)^pB\Phi^\dagger(Y)^{q}] = 2^{p+q}\tr [B^{p+q+2}] \ ,
%$$
%which fails if $p+q<1$.
%

\section{Duality, monotonicity and the Lieb Convexity Theorem}

The same ideas that we have developed concerning the monotonicity version of the Lieb Concavity Theorem may be developed with respect to the Lieb Convexity Theorem, Theorem~\ref{L2}, as we now explain.   We begin with the analog of Theorem~\ref{L2M}. Afterwards, we will deduce the analog of Theorem~\ref{Ep2}, and deduce its convexity analog.

\begin{thm}\label{Ep3A}   Let $0 < p < 1$, and let $\Phi: M_m(\C) \to M_n(\C)$ be a unital Schwarz map.  Then for any 
$A\in M_{n}^{++}(\C)$ and any invertible $B\in M_m(\C)$,
\begin{equation}\label{EpMon2}
 \tr[ (B^*A^{-p} B)^{1/(2-p)}]  \geq    \tr[ (\Phi^\dagger(B)^*\Phi^\dagger(A)^{-p} \Phi^\dagger(B))^{1/(2-p)}]   \ .
\end{equation}
\end{thm}

\begin{proof} Define $r := 1/(2-p)$ and note that $1/2< r < 1$.  Then by \eqref{minlem1}, 
\begin{eqnarray*}
&&( \tr[ (B^*A^{-p} B)^{1/(2-p)}])^{1/r} \\
&=& \left(\tr[ (B^*A^{-p} B)^r]\right)^{1/r}\\
 &=& \min\left\{ \tr[B^*A^{-p} B Y]\ :\ Y\in M_n^{++}(\C)\ ,\quad \tr[Y^{r/(r-1)}] =1\ \right\} \\
  &=& \min\left\{ \tr[B^*A^{-p} BY^{1-1/r}]\ :\ Y\in M_n^{++}(\C)\ ,\quad \tr[Y] =1\ \right\} \\
   &=& \min\left\{ \tr[B^*A^{-p} B Y^{p-1}]\ :\ Y\in M_n^{++}(\C)\ ,\quad \tr[Y] =1\ \right\} \\
   &\geq&   \min\left\{ \tr[  \Phi^\dagger(B)^*\Phi^\dagger(A)^{-p}  \Phi^\dagger(B)\Phi^\dagger(Y)^{p-1}]\ :\ Y\in M_n^{++}(\C)\ ,\quad \tr[Y] =1\ \right\}\\
    &\geq&   \min\left\{ \tr[  \Phi^\dagger(B)^*\Phi^\dagger(A)^{-p}  \Phi^\dagger(B)Z^{p-1}]\ :\ Z\in M_m^{++}(\C)\ ,\quad \tr[Z] =1\ \right\}\\
    &=&   \min\left\{ \tr[  \Phi^\dagger(B)^*\Phi^\dagger(A)^{-p}  \Phi^\dagger(B)Z]\ :\ Z\in M_m^{++}(\C)\ ,\quad \tr[Z^{r/(r-1)}] =1\ \right\}\\
    &=&  (\tr[ (\Phi^\dagger(B)^*\Phi^\dagger(A)^{-p} \Phi^\dagger(B))^{1/(2-p)}])^{1/r}\ .
\end{eqnarray*}
The first inequality is \eqref{lieb22}, and the second results from taking the minimum over a larger set. Now raise both sides to the $r$-th power. 
\end{proof}

\begin{remark}\label{tracial-rmk} Taking the limit $p \uparrow 1$ in \eqref{EpMon2}, we obtain that for all $A \in M_n^{++}(\C)$, all invertible $B\in M_n(\C)$ and all unital Schwarz maps $\Phi: M_m(\C) \to M_n(\C)$,
\begin{equation}\label{tracial}
\tr[ B^*A^{-1} B]  \geq    \tr[ \Phi^\dagger(B)^*\Phi^\dagger(A)^{-1} \Phi^\dagger(B))]  \ ,
\end{equation}
a special case of a result that was recently proved in \cite{CMH22}. Actually, they proved that \eqref{tracial} holds whenever $\Phi$ is Schwarz (and vice versa),  which we have already noted is an end-point case of Theorem~\ref{L2M}. Thus, one way to prove this special case (without assuming $\Phi$ to be unital), is to take the interpolation proof of  Theorem~\ref{L2M}
given in Section 5,   or Uhlmann's proof of Theorem~\ref{L1M} and  \cite[Lemma 1]{HP12quasi} to prove Theorem~\ref{L2M} (assuming $\Phi$ to be unital).  However, the simple and direct proof of a more general result in \cite{CMH22} is by far the most direct route to \eqref{tracial}. 

Notice that if we choose $A=\un$ in \eqref{tracial}, then 
\begin{equation}\label{tracialB}
\|\Phi^\dagger(\un)\|\tr[ B^*B]  \geq    \tr[ \Phi^\dagger(B)^* \Phi^\dagger(B))]  \ .
\end{equation}
\end{remark}

We now prove the extension of Theorem~\ref{L2M} corresponding to the full range of homogeneities in Theorem~\ref{L2}.

\begin{thm}\label{L2MB}   For all $s,t >0$ with $s+t \leq 1$, all $m,n\in \N$, all 
$X,Y \in M_n^{++}(\C)$ and all $K\in M_n(\C)$  and all  semiunital Schwarz maps $\Phi: M_m(\C)\to M_n(\C)$ such that $\Phi^\dagger(\un) \in M_m^{++}(\C)$, 
\begin{equation}\label{liebM22} 
 \tr[ \Phi^\dagger(K^*)\Phi^\dagger(Y)^{-s}  \Phi^\dagger(K)\Phi^\dagger(X)^{-t}]  \leq  \left(\frac{n}{m}\right)^{1-(s+t)} \tr[ K^*Y^{-s}  KX^{-t}]\ .
\end{equation} 
\end{thm}

\begin{proof}We proceed as in the proof of Theorem~\ref{L1MB}. As in that case, we may assume $s < 1-t$, and we define $r := s/(1-t)$ so that $0 < r < 1$. 
Since $\Phi$ is semiunital, $\frac{m}{n}\Phi^\dagger$ is unital and positive. Again, by \cite[Theorem 2.1]{Choi74} or \cite[Proposition 2.7.1]{Bhatia09pdm} ,  since $A \mapsto A^r$ is operator concave, 
$$\frac{m}{n}\Phi^\dagger(Y^r)  \leq \left(\frac{m}{n}\Phi^\dagger(Y)\right)^r\ ,$$
and hence by the fact that $A \mapsto A^{t-1}$ is operator decreasing,
\begin{equation}\label{monrel2}
\left(\frac{m}{n}\Phi^\dagger(Y^r) \right)^{t-1}  \geq \left(\frac{m}{n}\Phi^\dagger(Y) \right)^{-s}\ .
\end{equation}
By Theorem~\ref{L2M} and then \eqref{monrel2}
\begin{eqnarray*} 
\frac{m}{n} \tr[K^* Y^{-s} K X^{-t}] &=&  \frac{m}{n} \tr[K^* (Y^r)^{t-1} K X^{-t}] \\
&\geq & \frac{m}{n}\tr[\Phi^\dagger(K)^* \Phi^\dagger(Y^r)^{t-1} \Phi^\dagger(K) \Phi^\dagger(X)^{-t}]\\
&=& \tr\left[\left(\frac{m}{n}\Phi^\dagger(K)\right)^* \left(\frac{m}{n}\Phi^\dagger(Y^r)\right)^{t-1} \left(\frac{m}{n}\Phi^\dagger(K)\right) \left(\frac{m}{n}\Phi^\dagger(X)\right)^{-t}\right]\\
&\geq& \tr\left[\left(\frac{m}{n}\Phi^\dagger(K)\right)^* \left(\frac{m}{n}\Phi^\dagger(Y) \right)^{-s} \left(\frac{m}{n}\Phi^\dagger(K)\right) \left(\frac{m}{n}\Phi^\dagger(X)\right)^{-t}\right]\ ,
\end{eqnarray*}
and this is the same as \eqref{liebM22}. 
\end{proof}

\begin{remark} It is easy to see that taking $\Phi$ as in \eqref{embed}, one recovers Theorem~\ref{L2} from Theorem~\ref{L2MB}.
\end{remark} 

We now combine duality with Theorem~\ref{L2MB} to prove:

\begin{thm}\label{Ep3}  Let $0 \leq s < 1$. Let $\Phi$ be any sesquiunital Schwarz map from $M_m(\C)$ to $M_n(\C)$ for  $m,n\in \N$. 
Then for all $A \in M_n^{++}(\C)$ and all invertible $B\in M_n(\C)$,  and all
\begin{equation}\label{qbound}
\frac{1}{2-s} \leq q < 1\ ,
\end{equation}
we have
		\begin{equation}\label{ext2Z} 
	\left(\frac{n}{m}\right)^{(2-s)q-1}\tr[ (B^*A^{-s} B)^{q}] \geq \tr[ (\Phi^\dagger(B)^*\Phi^\dagger(A)^{-s}\Phi^\dagger(B))^q]\ .
	\end{equation} 
\end{thm}

\begin{proof}
 Let $0 < q < 1$. Then by \eqref{minlem1}, 
\begin{eqnarray*}
(\tr[ (B^*A^{-s} B)^{q}])^{1/q} &=& \max\{ \tr[B^*A^{-s}BY]\ :\ Y\in M_n^{++}(\C)\ ,\ \tr[Y^{q/(q-1)}] =1\ \}\\
&=& \max\{ \tr[B^*A^{-s}BY^{(q-1)/q}]\ :\ Y\in M_n^{++}(\C)\ ,\ \tr[Y] =1\ \}\ .
\end{eqnarray*}
In order to apply Theorem~\ref{L2MB}, we require $s + (1-q)/q  \leq 1$, and together with $q < 1$, this yields the requirement that \eqref{qbound} be satisfied.
Define $\delta := 2-s-1/q$. Then
with \eqref{qbound}  satisfied, Theorem~\ref{L2MB} and the previous calculation yield
\begin{eqnarray*}
&&\left(\frac{n}{m}\right)^{\delta} (\tr[ (B^*A^{-s} B)^{q}])^{1/q} \\
&= & \left(\frac{n}{m}\right)^{\delta}\max\{ \tr[B^*A^{-s}BY^{(q-1)/q}]\ :\ Y\in M_n^{++}(\C)\ ,\ \tr[Y] =1\ \}\\
&\geq &\max\{ \tr[\Phi^\dagger(B)^*\Phi^\dagger(A)^{-s}\Phi^\dagger(B)\Phi^\dagger(Y)^{(q-1)/q}]\ :\ Y\in M_n^{++}(\C)\ ,\ \tr[Y] =1\ \}\\
&\geq &\max\{ \tr[\Phi^\dagger(B)^*\Phi^\dagger(A)^{-s}\Phi^\dagger(B)Z^{(q-1)/q}]\ :\ Z\in M_n^{++}(\C)\ ,\ \tr[Z] =1\ \}\\
&=& (\tr[ (\Phi^\dagger(B)^*\Phi^\dagger(A)^{-s}\Phi^\dagger(B))^q])^{1/q}\ .
\end{eqnarray*}
Taking the $q$-th power of both sides yields \eqref{ext2Z}.
\end{proof}

\begin{remark} Note that for $q$ saturating the lower bound in \eqref{qbound}, \eqref{ext2Z} reduces to \eqref{EpMon2}, so that Theorem~\ref{Ep3} generalizes Theorem~\ref{Ep3A}. 
\end{remark}

Now take $n=2m$, $A_1,A_2\in M_m^{++}(\C)$ and invertible $B_1,B_2\in M_m(\C)$, and let
$$
A := \left[\begin{array}{cc} A_1 & 0\\ 0 & A_2\end{array}\right] \quad{\rm and}\quad  B := \left[\begin{array}{cc} B_1 & 0\\ 0 & B_2\end{array}\right]\ .
$$
Let $\Phi$ be given by \eqref{embed}. Then 
$$
\frac12 \tr[ (B^*A^{-s} B)^{q}] = \frac12 \left( \tr[ (B_1^*A_1^{-s} B_1)^{q}]  + \tr[ (B_2^*A_2^{-s} B_2)^{q}] \right)\ ,
$$
and
\begin{eqnarray*}
\left(\frac12\right)^{(2-s)q}\tr[ (\Phi^\dagger(B)^*\Phi^\dagger(A)^{-s}\Phi^\dagger(B))^q] &=&
\left(\frac12\right)^{(2-s)q} \tr[((B_1+B_2)^* (A_1+A_2)^{-s} (B_1+B_2))^q] \\
&=&  \tr\left[\left(\left(\frac{B_1+B_2}{2}\right)^* \left(\frac{A_1+A_2}{2}\right)^{-s} \left(\frac{B_1+B_2}{2}\right)\right)^q\right]\ .
\end{eqnarray*}

By Theorem \ref{Ep3} this proves:

\begin{cl}  Let $0 \leq s < 1$ and let $1/(2-s) \leq q  < 1$. Then the map
\begin{equation}\label{Ep7}
(A,B) \mapsto \tr[(B^* A^{-s}B)^q]
\end{equation}
is jointly convex on $M_n^{++}(\C)\times M_n(\C)$. 
\end{cl} 

A further corollary recover a result first proved in \cite{CL08}:

\begin{cl}  Let $1< p \leq 2$, and let $r \geq 1$. Then for all $B\in M_n(\C)$,  the map
\begin{equation}\label{Ep8}
A \mapsto \tr[(B^* A^p B)^{r/p}]
\end{equation}
is  convex on $M_n^{++}(\C)$.
\end{cl} 

\begin{proof}  First suppose that $1\leq r  <p$. As in  the proof of the Ando Convexity Theorem \cite{Ando1979concavity} as a direct consequence of the Lieb Convexity Theorem given in \cite{Carlen22review}, replace $B$ by $AB$ in \eqref{Ep7}, so that 
$\tr[(B^* A^{-s}B)^q]$ becomes $\tr[(B^* A^{2-s}B)^q]$. Define $p := 2-s$, and note that $1< p \le 2$.  Then $1/(2-s) \leq q <1$ becomes $1/p \leq q <1$.
Finally, replacing $q$ by $r/p$ yields \eqref{Ep8} in this case. 

The case $r \ge  p$ is elementary and does not depend on Theorem~\ref{L2M}. As in \cite[Lemma 2.1]{CL08} (see also
second part of the proof of Corollary~\ref{cor3.4}), one first shows by duality that with $q= r/p$, $f(A) := \|B^* A^{p}B\|_q$ is convex. Then $g(A) := \tr[(B^* A^{p}B)^q] = (f(A))^q$ is convex since $q \geq 1$. 
\end{proof}

\section{Interpolation proof of Theorems~\ref{L1M} and \ref{L2M}}

We give a simple proof of the monotonicity version of the Lieb Concavity Theorem for Schwarz maps using Lieb's original interpolation argument \cite{Lieb73WYD} adapted to this setting. This gives a third proof of this result and a further illustration of how direct proof of monotonicity can be both simpler and more general than direct proofs of concavity/convexity. 

\begin{proof}[Interpolation Proof of Theorem~\ref{L1M}]
	We don't need to reduce it to the case when $A=B$ that requires 2-positivity of $\Phi$. By approximation, we may assume that all the matrices $A,B,K$ and their images under $\Phi^\dagger$ are invertible. Denote $M:=\Phi^\dagger(A)^{p/2} K\Phi^\dagger(B)^{(1-p)/2}$. Then
	\begin{equation}
	\tr [K^\ast \Phi^\dagger(A)^{p} K \Phi^\dagger(B)^{1-p}]=\tr[ M^\ast M] \ .
	\end{equation}
	 Denote by $S$ the strip $\{z\in\mathbb{C}:0<\Re z<1\}$ and $\overline{S}$ its closure. Consider the following function
	\begin{equation*}
	f(z):=\tr\left[ \Phi\left( \Phi^\dagger (B)^{-\frac{1-z}{2}} M^\ast \Phi^\dagger(A)^{-\frac{z}{2}}\right)A^{z}\Phi\left( \Phi^\dagger(A)^{-\frac{z}{2}}M\Phi^\dagger(B)^{-\frac{1-z}{2}}\right)B^{1-z}\right]
	\end{equation*}
	on $\overline{S}$. Then it is analytic on $S$ and continuous on $\overline{S}$. It is also uniformly bounded on $\overline{S}$. In fact, for any $z=x+iy\in \overline{S}$ with $x\in [0,1]$ and $y\in \R$, by H\"older's inequality
	\begin{align*}
	|f(z)|\le &\|\Phi(Y(z)) A^z \Phi(X(z))B^{1-z}\|_1\\
	\le &\|\Phi(X(z))\|_2 \|\Phi(Y(z))\|_2\|A^{z}\|\|B^{1-z}\|\\
	\le &\|\Phi(X(z))\|_2  \|\Phi(Y(z))\|_2\|A\|^{x}\|B\|^{1-x},
	\end{align*}
	where $X(z):=\Phi^\dagger(A)^{-z/2}M\Phi^\dagger(B)^{(z-1)/2}$ and $Y(z):=X(\overline{z})^\ast=\Phi^\dagger (B)^{(z-1)/2} M^\ast \Phi^\dagger(A)^{-z/2}$. Since $\Phi$ is a Schwarz map, we have
	\begin{equation*}
	\|\Phi(X(z))\|_2^2
	= \tr \left[\Phi(X(z))^\ast \Phi(X(z))\right]
	\le \tr \left[\Phi(X(z)^\ast X(z))\right]
	\le \|\Phi^\dagger(\un)\|\tr \left[X(z)^\ast X(z)\right],
	\end{equation*}
	where the last inequality follows from H\"older's inequality. By definition, 
	\begin{equation*}
	\tr \left[X(z)^\ast X(z)\right]=\tr \left[ M^\ast \Phi^\dagger(A)^{-x}M \Phi^\dagger(B)^{x-1}\right] \ .
	\end{equation*}
	Similarly, we have 
	\begin{equation*}
	\|\Phi(Y(z))\|_2^2=\|\Phi(X(\overline{z}))^\ast\|_2^2\le \|\Phi^\dagger(\un)\|\tr \left[ M^\ast \Phi^\dagger(A)^{-x}M \Phi^\dagger(B)^{x-1}\right] \ .
	\end{equation*}
	Therefore we have proved
	\begin{equation*}
	\sup_{z\in \overline{S}}|f(z)|\le \|\Phi^\dagger(\un)\| \sup_{x\in [0,1]}\|A\|^x \|B\|^{1-x} \tr \left[ M^\ast \Phi^\dagger(A)^{-x}M \Phi^\dagger(B)^{x-1}\right]<\infty \ .
	\end{equation*}
	So $f$ is uniformly bounded on $\overline{S}$.  Note that by definition
	\begin{equation*}
	f(p)=\tr [\Phi(K)^\ast A^{p} \Phi(K)B^{1-p}] \ .
	\end{equation*}
	By the three-lines lemma, it remains to check that 
	\begin{equation}\label{ineq:boundary}
	\sup_{y\in\mathbb{R}}|f(k+iy)|\le \tr [M^\ast M], \quad {\rm for } \quad k=0,1 \ .
	\end{equation}
	By definition,
	\begin{align*}
	f(iy)=&\tr\left[ \Phi\left( \Phi^\dagger (B)^{-\frac{1-iy}{2}} M^\ast \Phi^\dagger(A)^{-\frac{iy}{2}}\right)A^{iy}\Phi\left( \Phi^\dagger(A)^{-\frac{iy}{2}}M\Phi^\dagger(B)^{-\frac{1-iy}{2}}\right)B^{1-iy}\right]\\
	=&\tr \left[  \Phi^\dagger (B)^{-\frac{1-iy}{2}} M^\ast \Phi^\dagger(A)^{-\frac{iy}{2}}\Phi^\dagger\left(X\right)\right]\\
	=&\tr \left[  \Phi^\dagger (B)^{\frac{iy}{2}} M^\ast \Phi^\dagger(A)^{-\frac{iy}{2}}\cdot \Phi^\dagger\left(X\right)\Phi^\dagger(B)^{-\frac{1}{2}}\right]
	\end{align*}
	where $X=A^{iy}\Phi\left( \Phi^\dagger(A)^{-iy/2}M\Phi^\dagger(B)^{(iy-1)/2}\right)B^{1-iy}$. By Cauchy--Schwarz inequality, we have
	\begin{align*}
	|f(iy)|^2\le \tr [M^\ast M] \tr [\Phi^\dagger(X)\Phi^\dagger(B)^{-1}\Phi^\dagger(X^\ast)]\ .
	\end{align*}
	Since $\Phi$ is Schwarz, by \eqref{tracial} or \cite[Theorem 4]{CMH22}
	\begin{align*}
	\tr [\Phi^\dagger(X)\Phi^\dagger(B)^{-1}\Phi^\dagger(X^\ast)]
	\le&\tr [XB^{-1}X^\ast]\\
	=&\tr \left[B\Phi\left( \Phi^\dagger(B)^{-\frac{1+iy}{2}}M^\ast\Phi^\dagger(A)^{\frac{iy}{2}}\right)\Phi\left( \Phi^\dagger(A)^{-\frac{iy}{2}}M\Phi^\dagger(B)^{-\frac{1-iy}{2}}\right)\right].
	\end{align*}
	Using again that $\Phi$ is Schwarz, this is bounded from above by
	\begin{align*}
	\tr \left[B\Phi\left( \Phi^\dagger(B)^{-\frac{1+iy}{2}}M^\ast M\Phi^\dagger(B)^{-\frac{1-iy}{2}}\right)\right]
	=\tr [M^\ast M]\ .
	\end{align*}
	So we have proved \eqref{ineq:boundary} for $k=0$. The proof of \eqref{ineq:boundary} for $k=1$ is similar.
\end{proof}

Using a similar argument, we can also give an interpolation proof of the monotonicity version of Lieb Convexity Theorem, Theorem~\ref{L2M}.
\begin{proof}[Interpolation Proof of Theorem~\ref{L2M}]
	Again, we may assume that all the matrices $X,Y,K$ and their images under $\Phi^\dagger$ are invertible. Denote $M:=Y^{(t-1)/2}KX^{-t/2}$. Then 
	$$\tr[K^\ast Y^{t-1}KX^{-t}]=\tr[ M^\ast M]\ . $$
	Let $S$ be the open strip as above. Consider the following function 
	\begin{equation*}
	g(z):=\tr\left[\Phi^\dagger\left(X^{\frac{z}{2}}M^\ast Y^{\frac{1-z}{2}}\right)\Phi^{\dagger}(Y)^{z-1}\Phi^\dagger\left(Y^{\frac{1-z}{2}}MX^{\frac{z}{2}}\right)\Phi^\dagger(X)^{-z}\right]
	\end{equation*}
	on $\overline{S}$. Then it is analytic on $S$ and continuous on $\overline{S}$. It is uniformly bounded on $\overline{S}$, which we now argue as above. For any $z=x+iy\in \overline{S}$ with $x\in [0,1]$ and $y\in \R$, by H\"older's inequality
	\begin{align*}
	|g(z)|\le &\|\Phi^\dagger(B(z))\Phi^{\dagger}(Y)^{z-1}\Phi^\dagger(A(z))\Phi^\dagger(X)^{-z}\|_1\\
	\le &\|\Phi^\dagger(A(z))\|_2 \|\Phi^\dagger(B(z))\|_2 \|\Phi^\dagger(Y)^{z-1}\| \|\Phi^\dagger(X)^{-z}\| \ ,
	\end{align*}
	where $A(z):=Y^{(1-z)/2}MX^{z/2}$ and $B(z):=A(\overline{z})^\ast=X^{z/2}M^\ast Y^{(1-z)/2}$.
	Note that $x-1,-x\le 0$, we get
	\begin{equation*}
	\|\Phi^\dagger(Y)^{z-1}\| \|\Phi^\dagger(X)^{-z}\| =\|\Phi^\dagger(Y)^{-1}\|^{1-x} \|\Phi^\dagger(X)^{-1}\|^{x} \ .
	\end{equation*} 
	Since $\Phi$ is Schwarz, by \eqref{tracialB}
	\begin{equation*}
	\|\Phi^\dagger(A(z))\|_2^2
	=\tr\left[\Phi^\dagger(A(z))^\ast \Phi^\dagger(A(z))\right]
	\le \|\Phi^\dagger(\un)\|\tr\left[A(z)^\ast A(z)\right]
	=\|\Phi^\dagger(\un)\|\tr\left[M^\ast Y^{1-x}MX^{x}\right] \ .
	\end{equation*}
	Similarly, 
		\begin{equation*}
	\|\Phi^\dagger(B(z))\|_2^2
	=\|\Phi^\dagger(A(\overline{z}))^\ast\|_2^2
	\le \|\Phi^\dagger(\un)\|\tr\left[M^\ast Y^{1-x}MX^{x}\right] \ .
	\end{equation*}
	So we have shown that 
	\begin{equation*}
	\sup_{z\in \overline{S}}|g(z)|\le \|\Phi^\dagger(\un)\| \sup_{x\in [0,1]}\|\Phi^\dagger(Y)^{-1}\|^{1-x} \|\Phi^\dagger(X)^{-1}\|^{x}\|\tr\left[M^\ast Y^{1-x}MX^{x}\right]<\infty \ .
	\end{equation*}
	Hence $g$ is uniformly bounded on $\overline{S}$.  Note that by definition 
	\begin{equation*}
	g(t)=\tr[\Phi^\dagger(K)^\ast \Phi^\dagger(Y)^{t-1}\Phi^\dagger(K)\Phi^\dagger(X)^{-t}]\ .
	\end{equation*}
	So by the three-lines lemma, it suffices to show that 
	\begin{equation}\label{ineq:boundary2}
	\sup_{y\in\mathbb{R}}|g(k+iy)|\le \tr [M^\ast M], \quad {\rm for } \quad k=0,1.
	\end{equation}
	When $k=0$, we have 
	\begin{align*}
	g(iy)&=\tr\left[\Phi^\dagger\left(X^{\frac{iy}{2}}M^\ast Y^{\frac{1-iy}{2}}\right)\Phi^{\dagger}(Y)^{iy-1}\Phi^\dagger\left(Y^{\frac{1-iy}{2}}MX^{\frac{iy}{2}}\right)\Phi^\dagger(X)^{-iy}\right]\\
	&=\tr\left[X^{\frac{iy}{2}}M^\ast Y^{\frac{1-iy}{2}} \Phi(A)\right]\\
	&=\tr\left[X^{\frac{iy}{2}}M^\ast Y^{-\frac{iy}{2}}\cdot Y^{\frac{1}{2}}\Phi(A)\right] \ ,
	\end{align*}
	where $A:=\Phi^{\dagger}(Y)^{iy-1}\Phi^\dagger\left(Y^{(1-iy)/2}MX^{iy/2}\right)\Phi^\dagger(X)^{-iy}$.
	By Cauchy--Schwarz inequality,
	\begin{equation*}
	|g(iy)|^2\le \tr[M^\ast M]\tr[\Phi(A)^\ast Y \Phi(A)]\ .
	\end{equation*}
	Since $\Phi$ is a Schwarz map, we have
	\begin{align*}
	\tr[\Phi(A)^\ast Y \Phi(A)]
	\le \tr\left[ A^\ast \Phi^\dagger(Y) A\right]
	=\tr\left[\Phi^\dagger(B)^\ast \Phi^\dagger(Y)^{-1}\Phi^\dagger(B)\right] \ ,
	\end{align*}
	where $B:=Y^{(1-iy)/2}MX^{iy/2}$. Again, since $\Phi$ is Schwarz,  by \eqref{tracial} or \cite[Theorem 4]{CMH22} that
	\begin{equation*}
	\tr\left[\Phi^\dagger(B)^\ast \Phi^\dagger(Y)^{-1}\Phi^\dagger(B)\right]
	\le \tr \left[B^\ast Y^{-1} B\right]
	=\tr [M^\ast M]\ .
	\end{equation*}
	So we have proved \eqref{ineq:boundary2} for $k=0$. The proof for $k=1$ is similar. 
\end{proof}

\end{document}